%% file: mz-complex-arxiv.tex
\documentclass[12pt]{amsart}
\usepackage{amsmath}
\usepackage{amsfonts,amsthm,amsopn,cite,mathrsfs,amssymb}
\usepackage{epsfig,verbatim}
\usepackage{graphicx}
\usepackage{subfigure}
\usepackage{xcolor, hyperref} 

\include{macros}
\newcommand{\mz}{Marcinkie\-wicz-\-Zyg\-mund}

\newcommand{\cro}{C_{\gamma /n}}
\newcommand{\br}{B_{1- \gamma /n}}

\newcommand{\gn}{1-\tfrac{\gamma}{n}}
\newcommand{\poln}{\cP _n}
\newcommand{\ano}{_{A^2}}
\newcommand{\hano}{_{H^2}}
\newcommand{\ar}{A_{\gamma /n}}

\begin{document}
\begin{abstract}
We study the relationship between sampling sequences in infinite
dimensional Hilbert spaces of analytic functions  and
Marcinkiewicz-Zygmund inequalities in subspaces of polynomials. We
focus on  the study of  the Hardy space and the Bergman 
space in one variable because they provide two  settings with a
strikingly different behavior.   
\end{abstract}

\title[Marcinkiewicz-Zygmund Inequalities]{Marcinkiewicz-Zygmund
  Inequalities for Polynomials in Bergmann and Hardy Spaces}
\author{Karlheinz Gr\"ochenig}
\address{Faculty of Mathematics \\
University of Vienna \\
Oskar-Morgenstern-Platz 1 \\
A-1090 Vienna, Austria}
\email{karlheinz.groechenig@univie.ac.at}

\author{Joaquim Ortega-Cerd\`a}
\address{Departament de Matem\`atiques i Inform\`atica \\
Universitat de  Barcelona \\
Gran Via de les Corts Catalanes, 585 \\
08007, Barcelona, Spain}
\email{jortega@ub.edu}

\date{}
\thanks{K.\ G.\ was
  supported in part by the  project P31887-N32  of the
Austrian Science Fund (FWF)}
 \maketitle

\section{Introduction}

\mz\ inequalities are finite-dimensional  models  for sampling in an infinite
dimensional Hilbert or Banach space of functions. Originally they
were studied in the context of interpolation by trigonometric
polynomials. They became
prominent in approximation theory, where they appear in quadrature rules and least square
problems,  and were usually studied in the context of   orthogonal
polynomials.

In an abstract setting one is given a reproducing kernel Hilbert space
$\cH $ on a set $S$ with reproducing kernel $k$ and a  sequence of
finite-dimensional subspaces $V_n$ such that $V_n \subseteq V_{n+1}$
and $\bigcup _n V_n$ is dense in $\cH $. Each $V_n$ comes with its own 
reproducing kernel $k_n$, which is the orthogonal projection of $k$. A
family of (finite) subsets $\Lambda _n \subseteq S$ is called a \mz\
family for $V_n$ in $\cH $, if there exist constants $A,B>0$,
the sampling constants, such that for all $n$ large, $n\geq n_0$, 
\begin{equation}
  \label{eq:k1}
\qquad   A \|p\|_{\cH }^2 \leq \sum _{\lambda \in \Lambda _n} \frac{|p(\lambda
    )|^2}{k_n(\lambda,\lambda )} \leq B \|p\|_{\cH }^2  \qquad \text{
    for all } p\in V_n  \, .
\end{equation}
Thus a \mz\ family comes with a sequence of \mz\ inequalities, which
are sampling inequalities for the finite-dimensional subspaces
$V_n$. 
The point of the definition is that the sampling constants are
independent of the subspace $V_n$. 

The diagonal of the reproducing kernel $k_n$ furnishes the most natural
choice of weights and goes back  to the corresponding notion of interpolating sequences in
reproducing kernel Hilbert 
spaces studied by Shapiro and Shields \cite{ShSh}. The weights are
intrinsic to the underlying spaces $V_n$. Another hint for using $k_n$
comes from frame theory: Since  $p(z) = \langle p, k_n (\cdot ,z)\rangle$ for $p\in
V_n $, the sampling inequality amounts to verifying that the
normalized reproducing kernels $k_n (\cdot
,\lambda)/k_n(\lambda,\lambda)^{1/2}$ form a  frame  in
$V_n $ all whose elements have   unit norm in  $\cH $. 

\mz\ inequalities have been studied in many different contexts, e.g.,
for trigonometric polynomials~\cite{Erd99,OS07}, for spaces of algebraic
polynomials with respect to some measure~\cite{Lub98,Lub99}, for spaces of
spherical harmonics on the sphere~\cite{Mar07,MOC10,MNW01}, for spaces of
eigenfunctions of the Laplacian on a compact Riemannian
manifold~\cite{OP12}, or even more generally for diffusion polynomials
on a 
metric measure space~\cite{FM11}. In sampling theory \mz\ inequalities
can be used to derive sampling
theorems for bandlimited functions~\cite{Gro99,Gro01b}.

In this paper we initiate the investigation of \mz\ inequalities for
polynomials in spaces of analytic functions. As the reproducing kernel
Hilbert space we take  either the Bergman
space $A^2(\bD )$ or the Hardy space $H^2(\bD )$  of analytic
functions on the unit disk $\bD $.  The natural finite-dimensional
subspaces will the 
family of the polynomials $\poln $ of degree $n$.
In  this context it is clear that an arbitrary set of at least $n+1$
distinct points yields  a sampling inequality
$ A\|p\|^2 \leq \sum _{\lambda \in \Lambda } |p(\lambda )|^2 \,
k_n(\lambda,\lambda) \inv  \leq B  \|p\|_{\cH } ^2$ for all $ p\in
\poln $. 
 The objective of \mz\ inequalities is to construct  a sequence of
finite sets $\Lambda _n$, such that  the constants $A,B$ are
independent of the degree $n$. The  game is therefore to diligently
keep track of the constants and show that they do not depend on the
degree. 

We will see that this problem is deeply  related to  sampling
theorems for the full space $A^2$ or $H^2$. 
 We say that $\Lambda \subseteq S $ is a sampling
set for $\cH$, if there exist constants $A,B>0$, such that
\begin{equation}   \label{origdef}
  A \|f\|_{\cH }^2 \leq \sum _{\lambda \in \Lambda } \frac{|f(\lambda
    )|^2}{k(\lambda,\lambda )} \leq B \|f\|_{\cH }^2  \qquad \text{
    for all } f\in \cH \, , 
\end{equation}
where now $k(z,w)$ is the reproducing kernel of $\cH $. In our case $\cH =
A^2(\bD)$ or $=H^2(\bD )$.

Both the Bergman space and the Hardy space are reproducing kernel
Hilbert spaces, in which the polynomials are dense and 
  $k_n(\lambda, \lambda)\to k(\lambda, \lambda)$ pointwise. However,
the underlying 
measures are different, and as a consequence the reproducing kernels
and the implicit metrics are different. We will see that these
differences imply a drastically different behavior  of \mz\
families. 
In the Bergman space the points of 
a \mz\ family 
will be ``uniformly'' distributed in the entire disk, in Hardy space
the points will cluster near the boundary of the disk. The
concentration will depend on the degree. It will therefore be
practical to introduce  a notation of the relevant  disks and annuli.
For a fixed parameter $\gamma >0$, we will write
\begin{equation}
  \label{eq:j1}
  \br = \{ z\in \bD : |z| < \gn \} = B(0,\gn ) 
\end{equation}
for the centered  disk of radius $\gn $,  and
\begin{equation}
  \label{eq:i2}
  \cro = \{z\in \bD : \gn \leq |z| <1 \} \, .
\end{equation}
for the annulus of width $\gamma /n $ at the boundary of $\bD $. 

Our main result for the Bergman space $A^2(\bD )$ with norm $\|f\|\ano
^2 = \tfrac{1}{\pi}\int _{\bD } |f(z)|^2 dz $, where $dz$ is the area
measure on  $\bD $, establishes a clear correspondence between
sampling sets 
for $A^2(\bD )$ and \mz\ families for the polynomials $\poln $ in
$A^2$ as follows.

\begin{tm} \label{tmintro1}
(i)    Assume that $\Lambda \subseteq \bD $ is a sampling set for $A^2(\bD
  )$. Then for $\gamma >0$ small enough, the sets $\Lambda _n =
  \Lambda \cap \br  $ form a \mz\ family for $\poln $ in
  $A^2(\bD )$.

(ii) Conversely, if 
  $(\Lambda _n)$ is a \mz\ family for the polynomials $\poln $ in
  $A^2(\bD )$, then every  weak limit of $(\Lambda
  _n)$ is a sampling set for $A^2(\bD )$. 
\end{tm}

See Section~\ref{sec:mzsamp} for the definition of a weak limit of
sets. 
The theorem shows that the construction of \mz\ families for the
Bergman space is on the same level of difficulty as the construction
of sampling sets for $A^2$. Fortunately, these sampling sets have been
characterized completely in the deep work of K.\
Seip~\cite{seip93}. Sampling sets are completely determined by a
suitable density, the Seip-Korenblum density. As a consequence of our
main theorem, one can now give many examples of \mz\ families for
$A^2$. 

By contrast, the Hardy space does not admit any sampling sequences. By
a theorem of P.\ Thomas~\cite{thomas98}(Props.~2 and~3), \emph{a function $f\in
H^2(\bD )$ satisfying $A \|f\| \hano ^2 \leq \sum _{\lambda \in \Lambda }
|f(\lambda )|^2 k(\lambda,\lambda )\inv \leq   B \|f\| \hano ^2$ must
be identical zero}.  Therefore there can be no analogue of
Theorem~\ref{tmintro1}(ii).

Despite the lack of a sampling theorem for $H^2(\bD )$ we can show the
existence of \mz\ families for polynomials with a different method.
The idea is to connect polynomials on the disc to  
polynomials on the torus in $L^2(\bT )$. By moving a \mz\ family for
 polynomials on $\bT $ into the interior of $\bD $, we
obtain a \mz\ family for polynomials in 
Hardy space. Since the problem on the torus is well
understood~\cite{OS07}, we can derive  a general construction of  \mz\ families
in $H^2(\bD )$.

\begin{tm} \label{tmintro2}
    Assume that the family $(\widetilde{\Lambda _n}) = \Big(\{e^{i \nu
      _{n,k}}: k= 1, \dots , L_n\} \Big) \subseteq \bT $
    is a \mz\ family for $\poln $  on the torus, i.e.,
    $$
    A \|p\|_{L^2(\bT )}^2 \leq   \sum _{k=1} ^{L_n}
      \frac{|p(e^{i \nu _{n,k} })|^2}{n} \leq B    \|p\|_{L^2(\bT )}^2
    $$
    for all  polynomials $p$  of degree $n$. 

    Fix $\gamma >0$ arbitrary, choose $\rho _{n,k} \in [\gn, 1)$
    arbitrary, and set $\Lambda _n = \{ \rho _{n,k} e^{i \nu _{n,k}}:
    k=1, \dots , L_n\} \subseteq \cro $ for $n\in \bN $. Then
    $(\Lambda _n)$ is a \mz\ family for $\cP _{n}$ in $H^2(\bT )$.  
    \end{tm}

 This result
 provides a systematic construction of examples of \mz\ families for
 Hardy space,  since \mz\ families for  polynomials on the torus can
 be characterized  almost completely by their density~\cite{OS07}. 
 \mz\ families on the torus and more generally of 
 orthogonal  polynomials have been studied intensely in approximation
 theory, see~\cite{Erd99,Gro99,Lub98,Lub99,MT00,Nev86}  for a sample
 of papers. 

The technical heart of the matter is, as so often in complex analysis, the investigation
and estimate of the reproducing kernels, namely the kernel $k(z,w)$  for the
entire space $A^2$ or $H^2$ and  the kernels $k_n(z,w)$   for the
polynomials of degree $n$. The guiding principle 
is to sample the polynomials in the region where the diagonals $k(z,z)$ and
$k_n(z,z)$ are comparable in size. One may call this region the
``bulk'' region.  For the Bergman space the bulk region is the
centered disk $\br $, because this is where the mass of polynomials of degree
$n$ is concentrated. For the Hardy space the bulk region is the
annulus $\cro $, as the $H^2$-norm sees only the boundary behavior of
functions in $H^2$.

Our main  insight may be relevant in other settings. For
instance, our main theorems can be extended to polynomials in weighted
Bergman spaces or in Fock
space~\cite{GOC21}. We expect a version of Theorem~\ref{tmintro1} to  hold  for Bergman space
on the unit ball in $\bC ^n$, though this will be more technical to
elaborate.

The paper is organized as follows: In Section~2 we treat the theory of
\mz\ families in the Bergman space $A^2(\bD )$ and in Section~3 we
treat the Hardy space $H^2(\bD )$. Each section starts with the
necessary background, the comparisons of  the various
reproducing kernels and the main contribution to the norms. Then we
formulate and prove the main results about \mz\ families. 

Throughout we will use the notation $\lesssim $ to abbreviate an
inequality $f \leq Cg$ where the constants is independent of the
essential input, which in our case will be the degree of the
polynomial. To indicate the dependence of the constant on some
 parameter $\gamma $, say, we will write $\lesssim \, _\gamma $. As
 usual, $f\asymp g$ means that both $f\lesssim g$ and $g\lesssim f$ hold.

\section{Bergmann space}

\subsection{Basic facts.}
The Bergmann space $A^2= A^2(\bD)$ consists of all analytic functions
on the unit disk $\bD $ 
with finite norm
\begin{equation}
  \label{eq:1}
  \|f\|_{A^2} = \Big( \frac{1}{\pi} \int _{\bD } |f(z)|^2 \, dz
  \Big)^{1/2} \, , 
\end{equation}
where $dz$ is the area measure on $\bD  $. 
For a detailed exposition of Bergman spaces we refer to the
excellent monographs~\cite{DS04,HKZ00}, both of which contain an
entire chapter on sampling in Bergman space.

The monomials $z\mapsto z^k$ are orthogonal with norm $\|z^k\|_{A^2}^2 =
\frac{1}{k+1}$. Consequently 
the norm of $f(z) =
\sum _{k=0}^\infty a_k z^k$ is 
\begin{equation}
  \label{eq:2}
  \|f\|_{A^2} ^2 = \sum _{k=0}^\infty |a_k|^2 \frac{1}{k+1} \, .
\end{equation}

Let $p(z) = \sum _{k=0}^n a_k z^k \in \poln $, then its norm on a disk $B_\rho $, $\rho <1$, is
given by
\begin{align*}
    \frac{1}{\pi} \int _{B_\rho } |p(z)|^2 \, dz &= \frac{1}{\pi} 
  \sum _{k,l=0}^n a_k\overline{a_l}   \int _{B_\rho } z^k \bar{z}^l \, dz
             \\
&= \frac{1}{\pi} 2\pi \sum _{k=0}^n |a_k|^2 \int _0^\rho r^{2k} rdr   \\
  &= \sum _{k=0}^n |a_k|^2 \rho ^{2k+2} \frac{1}{k+1}
\end{align*}
For  $p\in \poln $, we therefore have
\begin{equation}
  \label{eq:9}
    \frac{1}{\pi} \int _{B_\rho } |p(z)|^2 \, dz \geq \rho ^{2n+2}
    \|p\|_{A^2}^2 \, .
\end{equation}
To obtain a bound independent of $n$, we need to  choose $\rho _n$
such that $\rho _n^{2n} \geq A $ for all $n$. By picking  $\rho _n =
\gn $, we find 
\begin{equation}
  \label{eq:10}
e^{-2\gamma }\leq   (\gn )^{n} \leq e^{-\gamma} \qquad \forall n >
  2 \gamma \, .
\end{equation}
In the following we will  use  these inequalities abundantly. 

\begin{cor} \label{negl}
  If $p\in \poln $, then for every $\gamma >0$ and $n>2 \gamma$
$$
\frac{1}{\pi} \int _{C_{\gamma /n}} |p(w)|^2 \, dw \leq
\|p\|_{A^2}^2 \big(1- (\gn )^{2n+2} \big) \leq \|p\|_{A^2}^2 (1- \tfrac{1}{4}e^{-4\gamma }) \, .
$$
In particular, for $\epsilon >0$ there exists $\gamma >0$ (small
enough) such that
$$
\frac{1}{\pi} \int _{C_{\gamma /n}} |p(w)|^2 \, dw \leq
\epsilon \|p\|_{A^2}^2  \qquad \text{ for all } p\in \poln  \,
$$
independent of $n$. 
\end{cor}
\begin{proof}
With  \eqref{eq:9}  and the partition $\bD = \br \cup \cro $ we obtain  
  \begin{align*}
\frac{1}{\pi} \int _{C_{\gamma /n}} |p(w)|^2 \, dw &= \|p\|_{A^2}^2 -
 \frac{1}{\pi} \int _{\br} |p(w)|^2\, dw \\
  &  \leq \|p\|_{A^2}^2 \big(1 - (\gn )^{2n+2} \big) \leq 
\|p\|_{A^2}^2 (1-\frac{1}{4} e^{-4\gamma }) \, ,     
  \end{align*}
  since $(\gn )^{2n+2} \geq  e^{-4\gamma } (\gn )^2 \geq e^{-4\gamma }
  /4$ for $n\geq 2\gamma $. Also as $\gamma \to 0$, \\  $\inf _{n\in \bN } (\gn )^{2n+2}
  \to 1$. 
\end{proof}

\subsection{The Bergman kernels}

Since $\sqrt{k+1} \,z^k$ is an orthonormal basis for $\poln $ in
$A^2(\bD )$, the reproducing kernel of $\cP _n$ in $A^2$ is given by
\begin{align}  
  k_n(z,w) = \sum _{k=0}^n
             (k+1) (z\bar{w})^k = \frac{1+(n+1)(z\bar{w})^{n+2} -
    (n+2)(z\bar{w})^{n+1}}{(1-z\bar{w})^2} \, .  \label{eq:3}
\end{align}
As $n\to \infty $, the kernel tends to the Bergman  kernel of
$A^2$,
\begin{equation}
  \label{eq:4b}
k(z,w) =   \frac{1}{(1-z\bar{w})^2} \qquad z,w \in \bD \, .
\end{equation}

We first compare these kernels in two regimes, namely the ``bulk''
regions $\br $ and the boundary region $\cro $. 
\begin{lemma}
  \label{wei1}
  Let $k_n(z,w)$ be the reproducing kernel of $\poln$ in $A^2$ and
  $\gamma >0$ be arbitrary. 

  (i) If $|z| \leq \gn$, then $k_n(z,z) \asymp k(z,z)$  for $n$ large
  enough,  $n\geq n_\gamma $, precisely 
  \begin{equation}
    \label{eq:5}
    c_\gamma k(z,z) \leq k_n(z,z) \leq k(z,z) = \frac{1}{(1-|z|^2)^2}
    \, ,
  \end{equation}
where $c_\gamma >0$ can be chosen as  $c_\gamma =
1-e^{-2\gamma}(1+2\gamma )$.

  (ii) If $\gn \leq |z| <1$ and $n>\max (\gamma , 3) $, then $k_n(z,z)
  \asymp n^2$, precisely, 
  \begin{equation}
    \label{eq:6}
    \frac{e^{-4\gamma } }{4} n^2 \leq k_n(z,z) \leq n^2 \, .
  \end{equation}
\end{lemma}
Thus inside the disk $\br$ the reproducing kernel for $\poln $ behaves like the
reproducing kernel of the full space $A^2$, whereas in the annulus it
behaves like $n^2$ and is much smaller than $k(z,z)$ near the boundary
of $\bD $. 
\begin{proof}
  (i) It is always true that $k_n(z,z) \leq k(z,z) =
  \frac{1}{(1-|z|^2)^2}$. For the lower bound, let $r=|z|$ and observe
  that
  \begin{equation}
    \label{eq:22a}
  k_n(z,z) = k(z,z) \big( 1 - (n+2)r^{2n+2} + (n+1) r^{2n+4} \big) =
  k(z,z) ( 1 - q(r)) \, .  
  \end{equation}
    Since $q(r) = (n+2)r^{2n+2} - (n+1) r^{2n+4}$ is increasing on $(0,1)$ and $r\leq \gn $, we need an
  upper   estimate for $q$ at $\gn $. With the help of \eqref{eq:10}
  and some algebra we write $q$ as 
  \begin{align*}
    q(\gn ) & = (\gn )^{2n+2} \big( n+2 - (n+1)(\gn )^2 \big) \\
            & =  (\gn ) ^{2n } \big( \gn )^2 \, \big( 1 + 2\gamma + \tfrac{\gamma}{n}
      (2-\gamma ) - \tfrac{\gamma ^2}{n^2} \big) \\ 
&= (\gn )^{2n} \Big(1+2\gamma - \frac{5\gamma ^2}{n} +
                                                       \frac{\alpha_\gamma}{n^2} + \frac{\beta _\gamma }{n^3}  \Big) \, 
  \end{align*}
  for some constants $\alpha _\gamma , \beta _\gamma $.
Using \eqref{eq:10} and a sufficiently large $n$, $n\geq n_\gamma $
say, the  final estimate for $q$ is
$$
q(\gn ) \leq  e^{-2\gamma } (1+2\gamma ) < 1 \, , \text{ for } n \geq
n_\gamma \, .
$$
Combined with \eqref{eq:22a} we have the lower estimate
$$
k_n(z,z) \geq (1- e^{-2\gamma } (1+2\gamma )) k(z,z) = c_\gamma k(z,z)
\,
$$
for $|z| \leq \gn $ and $n\geq n_\gamma $. 

  (ii) If $\gn \leq |z| <1$, then
  $$
  k_n(z,z) = \sum _{k=0}^n |z|^{2k} (k+1) \leq \sum  _{k=0}^n (k+1) =
  \frac{(n+1)(n+2)}{2} \leq n^2 \, ,
  $$
for $n\geq 3$,   and, for $n\geq 2\gamma $, 
\begin{equation*}
  k_n(z,z) = \sum _{k=0}^n |z|^{2k} (k+1) \geq \sum _{k\geq n/2}^n
               (\gn )^{2n} \, \frac{n}{2} 
               \geq e^{-4\gamma } \frac{n^2}{4} \, .  
             \end{equation*}
\end{proof}

 We will  also need some information about the  behavior
 of the reproducing kernels $k_n$ near the diagonal. It will be convenient to use the 
 normalized reproducing kernel $\kappa_n(z,w)$  for $\poln $ at the
 point $w$, i.e.,  $\kappa_n(z,w) =
 k_n(z,w)/\sqrt{k_n(w, w)}$. It satisfies  $\| \kappa_n(\cdot,w)\| =
 1$, and   as we have observed in~\eqref{eq:5}, 
 if $w \in \br $ then $\kappa_n(w,w)
 \asymp \frac1{1-|w|^2}$.  The following lemma collects the
 properties of $\kappa _n$ near the diagonal in the different regimes in $\bD $. 
\begin{lemma} \label{neardiag}
(i)  
There is a constant $\gamma_0$ such that 
for all $w\in B_{1-\gamma_0/n}$ and $z\in B(w, 0.5(1-|w|^2))$ we have
$$
 \frac{1}{4(1-|w|^2)}|\leq |\kappa_n(z,w)| \leq 
\frac{9}{4(1-|w|^2)} \, .
$$

(ii) 
For every  $\gamma>0$  there
are $K>0$ and $\epsilon > 0$ depending only on $\gamma$ such that
$$
\frac{n}{K} \le  |\kappa_n(z,w)| \le K n \, \qquad \text{ for all } w\in \cro , z\in B(w, \epsilon /n)
$$
and $n>\gamma $. 
\end{lemma}
\begin{proof}
  (i)  If $z\in B(w, 0.5(1-|w|^2))$ then
  \begin{equation}
    \label{eq:num}
  |1-z\bar w| = |1-|w|^2 - \bar{w} (z-w)| \geq (1-|w|^2) - |\bar{w}
  (z-w)| \geq \tfrac{1}{2} ( 1-|w|^2) \, ,  
  \end{equation}
  and likewise $  |1-z\bar w| \leq \tfrac{3}{2} ( 1-|w|^2)$. To obtain
similar bounds  for the
numerator, it suffices to  prove that 
 $$
 |(n+2)(z\bar{w})^{n+1}- (n+1)(z\bar{w})^{n+2} | < \tfrac{1}{2} \, .
 $$
 for $\gamma $ sufficiently large.  We replace $n+1$ by $n$ and  rewrite this expression
 as
 \begin{align*}
   |(n+1)(z\bar{w})^{n}- n(z\bar{w})^{n+1}| & = \Big| n (z\bar{w})^{n}
    \Big( \tfrac{n+1}{n} - z\bar{w} \Big) \Big| \\ 
&\leq n |w|^n \Big( |1-z\bar{w}| + \tfrac{1}{n} \Big) \\
&\leq n |w|^n \tfrac{3}{2} (1-|w|^2) + |w|^n \, . 
 \end{align*}
As the map $r \to r^n(1-r^2)$ is
increasing on $[0, (\frac{n}{n+2})^{1/2}]$, the maximum of this 
expression is taken at $r=\gn $, so that we obtain the estimate
$$
   |(n+1)(z\bar{w})^{n}- n(z\bar{w})^{n+1}| \leq 3n (\gn )^{n}
   \frac{\gamma}{n} + (\gn )^{n} \leq 3 e^{-\gamma } \gamma +
   e^{-\gamma } < 1/2 \, ,
  $$
for $\gamma $ large enough. In fact, we may take  $\gamma \geq 3$.  

(ii)  We know from  \eqref{eq:6} that  $k _n(w,w) \asymp  n^2$. 
Furthermore 
$$
|\kappa_n'(z,w)| =  \frac{1}{k_n(w,w)^{1/2}} \Big| \frac{\partial k_n(z,
  w)}{\partial z}\Big| \lesssim  \frac{1}{n}\, \Big| \sum_{k= 1}^n k(k+1) (z\bar w)^{k-1}\bar
w\Big| \le C n^2
$$
Thus if $z\in B(w, \epsilon/n)$ then $\kappa_n(z) \ge
\frac{e^{-2\gamma}}{2}n - C \epsilon n$. Now  take $\epsilon \le
\frac{Ce^{-2\gamma}}{4}$ and $K = \frac{e^{-2\gamma}}{4}$ and (ii) 
follows. 
\end{proof}

\color{black}

\subsection{Separation and Carleson-type conditions}

We first study the upper estimate of the \mz\ inequalities and derive
a geometric description. 

Let $d(z,w) = \big| \frac{z-w}{1-z\bar{w}}\big|$ be the pseudohyperbolic metric on
$\bD $. We denote $\Delta (w,\rho ) = \{ z\in \bD : d(z,w) <\rho \}$
the hyperbolic disk in $\bD $. While $\Delta (w,\rho )$ is also a
Euclidean disk (albeit with a different center and radius), it will be
more convenient for us to compare it to a Euclidean disk with the same
center $w$. In fact, we have the following inclusions
\begin{equation}
  \label{eq:n2}
  B(w, \frac{\rho}{1+\rho} (1-|w|^2)) \subseteq \Delta (w,\rho )
  \subseteq B(w, \frac{\rho}{1-\rho} (1-|w|^2)) \, ,
\end{equation}
where the latter inclusion holds for $\rho <1/2$.

A set $\Lambda \subseteq \bD $ is called uniformly discrete,  
 if there is a
$\delta ' >0$ such that $d(\lambda ,\mu )\geq \delta ' $ for all $\lambda
,\mu \in \Lambda , \lambda \neq \mu $.
In view of \eqref{eq:n2} this is equivalent  to the fact
that the Euclidean  balls $B(\lambda , \delta  (1-|\lambda |)),
\lambda \in \Lambda $ are disjoint in $\bD $ for some $\delta >0$. We
refer to this condition as $\Lambda $ being $\delta $-separated. 

 In $A^2(\bD )$ the upper estimate of the sampling inequality is
 characterized by the following geometric condition. See ~\cite[Sec.~2.11, Thm.~15]{DS04} and \cite{seip93}. 

\begin{prop} \label{geomcarleson}
  For $\Lambda \subseteq \bD $ the following conditions are
  equivalent:

  (i) The inequality $\sum _{\lambda \in \Lambda } \frac{|f(\lambda
    )|^2}{k(\lambda,\lambda )} \leq B \|f\|^2 \ano $ holds for all $f\in
A^2(\bD )$.

(ii) $\Lambda $ is a finite union of uniformly discrete sets.

(iii) $\sup _{w\in \bD } \big(\Lambda \cap \Delta (w,\rho )\big)  <
\infty $ for some (hence all) $\rho \in (0,1)$. 
\end{prop}
Condition (i) is often formulated  by saying that the measure $\sum
_{\lambda \in \Lambda } k(\lambda,\lambda ) \delta _\lambda $ is a
Carleson measure for $A^2$. 

If we add the (much more difficult) lower sampling inequality to the
assumptions, we also have  the following lemma of Seip~\cite{seip93} (see Lemma~5.2 and Thm.\ 7.1)

\begin{lemma} \label{relsep}
  If $\Lambda $ is a sampling set for $A^2(\bD )$, then $\Lambda $
  contains a uniformly discrete  set $\Lambda ' \subseteq \Lambda $ that is also
  sampling for $A^2(\bD )$. 
\end{lemma}

The proof of the implication $(ii) \Rightarrow (i)$ yields a local
version of the Bessel inequality that   will be  needed. 
\begin{lemma}
  \label{boundary}
  Let $\Lambda $ be a  $\delta $-separated set 
  and let $\gamma >0$. Then there exists a
  constant $C = C(\delta )$ and $\gamma ' >\gamma $, such that
  \begin{equation}
\label{eq:7}
\sum _{\lambda \in \Lambda \cap \cro } \frac{|f(\lambda ) |^2}{k(\lambda , \lambda
  )} \leq C \, \int _{C_{\gamma'/n}} |f(w)|^2 \, dw \, \qquad \text{
  for all } f\in A^2 \, .
\end{equation}
The constants depend only on the separation via  $\gamma ' = (1+\delta )
\gamma $ and $C = \frac{4}{\pi \delta ^2}$. 
\end{lemma}
\begin{proof}
See ~\cite{DS04}, Sec.~2.11, Lemma 14. For
completeness we include the proof. 

  By assumption    the Euclidean balls $B_\lambda = B(\lambda , \delta (1-|\lambda
|)) \subseteq \bD $ for $\lambda \in \cro $ are disjoint. 
Since $|B_\lambda | k(\lambda,\lambda ) = \pi \delta ^2
(1-|\lambda|)^2 (1-|\lambda |^2)^{-2} \geq \pi \delta ^2/4$, 
 the submean-value property for $|f|^2$ yields the estimate
$$
\frac{|f(z)|^2}{k(\lambda,\lambda)}  \leq \frac{1}{k(\lambda,\lambda) |B_\lambda| } \int _{B_\lambda } |f(w)|^2 \, dw \leq
  \frac{4}{\pi \delta ^2} \int _{B_\lambda } |f(w)|^2 \, dw
\, . 
$$
By summing over $\lambda \in \Lambda \cap \cro $ and using the disjointness of the
disks $B_\lambda $ we obtain that
\begin{align*}
\sum _{\lambda \in \cro } \frac{|f(\lambda ) |^2}{k(\lambda , \lambda
  )} & \leq   \frac{4}{\pi \delta ^2}  \, \sum
  _{\lambda \in \cro }  \int
       _{B_\lambda} |f(w)|^2 \, dw 
  = \frac{4}{\pi \delta ^2} \,  \int
       _{\bigcup _{\lambda \in \cro } B_\lambda} |f(w)|^2 \, dw \, .
\end{align*}
Since $\mathrm{dist}\,  (0,B_\lambda ) = |\lambda | - \delta
(1-|\lambda |) = |\lambda | (1+\delta )  - \delta \geq (\gn ) (1+\delta )
- \delta = 1-\frac{(1+\delta )\gamma }{n} = 1-\frac{\gamma'}{n}$, the
disks $B_\lambda $ are contained in the annulus $C_{\gamma' /
  n}$. Since they are disjoint, we obtain 
$$
\int _{\bigcup _{\lambda \in \cro } B_\lambda} |f(w)|^2 \, dw \leq
       \int _{C_{\gamma' /n}} |f(w)|^2 \, dw \, .
       $$
       If $\Lambda $ is a union of $K$  separated sets $\Lambda _m$,
       we apply the above argument to each  $\Lambda _m$ and then
       obtain the constant $C = \frac{4K}{\pi \delta ^2}$. 
\end{proof}

Next we develop a  geometric description for the upper \mz\
inequalities in $A^2$ that is similar to
Proposition~\ref{geomcarleson}.   For this  we estimate the number of
points of a \mz\ family in the 
  relevant regions of the disk, namely in the bulk $\br $, the annulus
  $\cro$, and in the cells $B(w,\epsilon/n)\cap \bD$ for $w$ near the boundary
  of $\bD $.

\begin{prop} \label{boundarycount}
  Assume that $(\Lambda _n)$ satisfies the upper  \mz\ inequalities
  \eqref{eq:k1}  for $\poln $ in $A^2(\bD
  )$.

  (i) Then for every $\gamma >0$
  \begin{equation}
    \label{eq:13}
    \# (\Lambda _n \cap \cro ) \leq C n \, .
  \end{equation}

  (ii) There are  $\gamma_0>0$ ($\gamma _0 \approx 3$)  and
 $C>0$ such that  
 \[
  \#(\Lambda_n \cap B(z, 0.5(1-|z|))) \le C, \qquad \forall z\in B_{1-\gamma_0/n}.
 \]
As a consequence, $\Lambda _n \cap B_{1-\gamma _0/n} $ is a disjoint union of at
most $C$ uniformly discrete subsets with a separation $\delta $
independent of $n$. 

 (iii) For every $\gamma >0$ there is $\epsilon >0$ such that
 $$
  \#(\Lambda_n \cap B(z, \tfrac{\epsilon}{n})) \le C, \qquad \forall
  z\in \cro \, .
  $$
 The  constants  depend only on $\gamma $ and the upper \mz\
  bound $B$.
\end{prop}

\begin{proof}
In each case   we test the upper \mz\ inequalities against a suitable
polynomial in $\poln $.

(i)   Choose the monomial $p(z) = z^n \in \poln
  $ with norm $\|p\| \ano ^2 = \tfrac{1}{n+1} $. Since for $\lambda
  \in \cro$,  $|p(\lambda)| \asymp _\gamma 1$ 
  by  \eqref{eq:10} and $ k_n(\lambda ,\lambda )  \asymp _\gamma n^2$
  by \eqref{eq:6},  we obtain
  \begin{align*}
    \frac{1}{n^2} \# (\Lambda _n \cap \cro ) &\asymp _\gamma 
   \sum _{\lambda \in   \Lambda _n \cap \cro } \frac{|p(\lambda)|^2}{k_n(\lambda,\lambda)}
\leq   B \|p\| ^2 \ano =  \frac{B}{n+1}\, .
  \end{align*}
  This implies $  \# (\Lambda _n \cap \cro ) \lesssim  n $.

  (ii) We choose 
  $\kappa_n (\cdot ,z) \in \poln $.  For $z\in \br$ and $\lambda \in
  B(z, 0.5(1-|z|))$ we have 
 $|\kappa
  _n(\lambda, z)|^2 \asymp (1-|z|^2)\inv$   by
  Lemma~\ref{neardiag}(i),   and $k_n(\lambda,\lambda) \asymp (1-|\lambda|^2)\inv
\asymp (1-|z|^2)\inv $  by  
Lemma~\ref{wei1}.     We conclude  that $\frac{|\kappa
  _n(\lambda, z)|^2}{k_n(\lambda,\lambda)} \geq C$ for some constant,
and thus 
$$
C    \#(\Lambda_n \cap B(z, 0.5(1-|z|))) \leq  \sum _{\lambda \in B(z,
  0.5(1-|z|))} \frac{|\kappa
  _n(\lambda, z)|^2}{k_n(\lambda,\lambda)} \leq B \|\kappa _n (\cdot ,
z)\| \ano ^2 = B \, .
$$
The   relative separation follows as 
in~\cite{DS04}, Section 2.11, Lemma 16. 

(iii) Again we choose $\kappa _n(\cdot , z)$, this time  for $z\in
\cro $. By Lemma~\ref{neardiag} 
$|\kappa _n (\lambda,z)|^2 \asymp n^2$ for $\lambda \in
B(z,\epsilon/n)$. Likewise $k_n(z,z) \asymp n^2$ for $z\in \cro $ by 
Lemma~\ref{wei1}. Thus
$$
C \#\big( \Lambda_n \cap B(z, \tfrac{\epsilon}{n})\big) \leq  \sum _{\lambda \in
  \Lambda _n \cap  B(z, \frac{\epsilon}{n}) } \frac{|\kappa
  _n(\lambda, z)|^2}{k_n(\lambda,\lambda)} \leq B \|\kappa _n (\cdot ,
z)\| \ano ^2 = B \, .
$$
\end{proof}

As a consequence we see that the cardinality of a \mz\ family for
polynomials obeys the correct order of growth. 

\begin{cor}
 If $(\Lambda_n)$ satisfies the upper Marcinkiewicz-Zygmund inequality, then 
 $\# \Lambda_n \lesssim n$.
\end{cor}
\begin{proof}
 Choose $\gamma$ large  enough as in
 Lemma~\ref{boundarycount}(ii). We cover $\br = \Delta (0, \gn )$ with
 hyperbolic disks $\Delta (z_j,\frac{1}{3}) \subseteq B(z_j, 0.5
 (1-|z_j|))$, such that the disks $\Delta (z_j,\frac{1}{6})$ are
 disjoint.  Since the hyperbolic area of $\br $ is
 $\int_{B_{1-\gamma/n}} \frac{dz}{(1-|z|^2)^2} \leq n/\gamma$, we need
 at most $c n/\gamma $ disks (where $c\inv $ is the hyperbolic area of
 $\Delta (z_j,\frac{1}{6})$).  By Lemma~\ref{boundarycount}(ii) every
 hyperbolic disk $ B(z_j, 0.5
 (1-|z_j|))$ contains $C$ points, so that 
$ \# (\Lambda_n\cap B_{1-\gamma/n}) \leq C n$ for a suitable constant.  
By  Lemma~\ref{boundarycount}(i) 
$\#(\Lambda_n \cap
C_{\gamma/n}) \lesssim n $ as well.
\end{proof}

 

We can now formulate  a geometric description of the upper \mz\
inequality~\eqref{eq:k1}, that is,  the Bessel inequality of the
normalized reproducing kernels. 
\begin{tm}\label{besselgeom}
For a family $(\Lambda_n)\subseteq \bD $ the following conditions are
equivalent:

(i) For all $n\geq n_0$
\[
          \sum_{\lambda\in\Lambda_n}
          \frac{|p(\lambda)|^2}{k_n(\lambda, \lambda)} \le B
          \|p\|^2,\qquad \forall p\in \mathcal P_n \, . 
\]

(ii)  There  are  $\gamma>0$ and $C>0$ such that $(\Lambda_n)$ satisfies
\begin{align}
 \# \big(\Lambda_n \cap B(w, 0.5(1-|w|))\big) &\le C \qquad \forall w\in \br  \, , \label{eq:n7}\\
 \#(\Lambda_n \cap B(w, 1/n)) & \le C \qquad \forall w\in \cro \, . \label{eq:n8}
\end{align}

\end{tm}

\begin{proof}
(i) $\, \Rightarrow \, $ (ii): This is
Proposition~\ref{boundarycount}.

(ii) $\, \Rightarrow \, $ (i): We write  
$$
   \sum_{\lambda\in\Lambda_n}
          \frac{|p(\lambda)|^2}{k_n(\lambda, \lambda)} = \sum
          _{\lambda \in \Lambda _n \cap \br} \dots + \sum _{\lambda
            \in \Lambda _n \cap \cro } \dots
          $$
and control each term with the appropriate geometric condition.

On $\br $ we may replace $k_n$ by $k$ (Lemma~\ref{wei1}). Since
$\Lambda _n \cap \br $ is a union of at most $C$ uniformly discrete
sets with fixed separation independent of $n$ by our
assumption~\eqref{eq:n7}, Proposition~\ref{geomcarleson}  --- or the appropriate
version of Lemma~\ref{boundary} --- implies that
\begin{align*}
  \sum _{\lambda \in \Lambda _n \cap \br }
  \frac{|p(\lambda)|^2}{k_n(\lambda, \lambda)} & \lesssim _\gamma   \sum _{\lambda \in \Lambda _n \cap \br }
                                                 \frac{|p(\lambda)|^2}{k(\lambda, \lambda)} \\
  &\lesssim \|p\|\ano ^2 \, .
\end{align*}

  
On $\cro $ we have $k_n(\lambda ,\lambda) |B(\lambda, 1/n)| \asymp 1
$ by~\eqref{eq:6}, therefore using the submean-value property we
obtain
\begin{align*}
\sum _{\lambda \in \cro } \frac{|p(\lambda)|^2}{k_n(\lambda,\lambda)}
  &\leq \sum   _{\lambda \in \cro } \frac{1}{k_n(\lambda,\lambda)
    |B(\lambda,1/n)|} \int _{B(\lambda , 1/n)} |p(z)|^2 \, dz \\
  & \lesssim \int   \Big( \sum _{\lambda \in \Lambda
    _n \cap \cro } \chi _{B(\lambda,1/n)} (z ) \Big) |p(z)|^2 \, dz
    \, .
\end{align*}
The integral is  over the slightly
larger annulus $C_{(\gamma +1)/n}\supseteq \bigcup _{\lambda \in \cro
} B(\lambda,1/n)$,  and by assumption \eqref{eq:n8} the
sum in the integral is bounded, whence  $\sum _{\lambda \in \cro } \frac{|p(\lambda)|^2}{k_n(\lambda,\lambda)}
 \lesssim  \|p\|\ano ^2$. 
\end{proof}

In   a more general formulation   we can use measures in condition (i) 
instead of discrete samples. This motivates
the following definition. 
\begin{df}
We say that a sequence of measures $\{\mu_n\}$ defined on the unit
disk form a Carleson sequence for polynomials in the Bergman space if
there  is a constant $C>0$ such that 
\begin{equation}
  \label{eq:carl}
 \int_{\mathbb D} |p(z)|^2\, d\mu_n(z) \le C  \int_{\mathbb D} |p(z)|^2\, dz
 \qquad \text{ for all } p\in \poln
\end{equation}
uniformly in $n$. 
\end{df}
Thus a sequence  $(\Lambda_n)\subseteq \bD $ satisfies the upper
inequality in \eqref{eq:k1} if and only if the measures $\mu_n =
\sum_{\lambda \in \Lambda_n} \frac{\delta_\lambda}{k_n(\lambda,
  \lambda)}$ form a Carleson sequence.

\begin{tm}\label{carleson}
 A sequence of measures $\mu_n$ is a Carleson sequence for polynomials
 in the Bergman space if and only if there exist $\gamma> 0$ and 
 $C>0$ such that 
 \begin{align*}
 \mu_n(B(z,0.5(1-|z|^2))) &\le C (1-|z|^2)^2 \qquad \forall z\in
   B_{1-\gamma/n}, \\  
   \mu_n(B(z, 1/n)) &\le C/{n^2}  \qquad \qquad \quad \forall z\in C_{\gamma/n}.
 \end{align*}
\end{tm}

The proof is similar to the proof of Theorem~\ref{besselgeom}. As
we will not need the general statement about Carleson sequences, we
omit the details.

\vspace{3mm}
\noindent \textbf{Remark.}
   The geometric condition in Theorems~\ref{besselgeom} and
  \ref{carleson} can be more concisely stated  in terms of the first
  Bergman metric $\rho_n = k_n(z,z) ds^2$, see \cite[p.32]{Bergman}
  instead of the Euclidean metric. This is not to be confused with the
  more usual second Bergman metric given by $\Delta \log k_n(z,z)
  ds^2$. The reformulation is as follows: if $D_n(z,r)$ denotes a disk
  in the first Bergman metric $\rho_n$, then a sequence of measures
  $\{\mu_n\}$ is a Carleson sequence if and only if there are  $r> 0$
  and $C>0$ such that  
 $\mu_n(D_n(z,r)) \le C |D_n(z,r)|$, for all $n>0$ and for all $z\in
 \bD$, where $|D(z,r)|$ is the Lebesgue measure of $D(z,r)$. The
 conditions of 
 (ii) in Theorem~\ref{besselgeom} are then equivalent to  $\# (\Lambda _n \cap
 D_n(z,r)) \leq C$ for all $z\in \bD $.

 This reformulation  is of course similar to the characterization of
 the Carleson measures for the Bergman space obtained in
 \cite{Hastings}. The geometric description is $\mu(B(z,0.5(1-|z|)))
 \lesssim |B(z, 0.5(1-|z|))|$ for all $z\in \bD$.  

\subsection{Sampling implies \mz\ inequalities}
After these preparations we can show that every sampling set for
$A^2(\bD )$ generates  a \mz\ family for the polynomials. 

\begin{tm} \label{samptomz}
  Assume that $\Lambda \subseteq \bD $ is a sampling set for $A^2(\bD
  )$. Then for $\gamma >0$ small enough, the sets $\Lambda _n =
  \Lambda \cap \br $  form a \mz\ family for $\poln $ in
  $A^2(\bD )$. 
\end{tm}

\begin{proof}
  The assumption means that there exist $A,B>0$, such that
  \begin{equation}
    \label{eq:12}
    A \|f\|_{A^2}^2 \leq  \sum _{\lambda \in \Lambda }
    \frac{|f(\lambda )|^2}{k(\lambda , \lambda )}   \leq B \|f\|_{A^2}^2 \quad \text{ for all
    } f \in A^2 \, .
  \end{equation}

 Since $\Lambda $ is sampling for $A^2$, by Lemma~\ref{relsep} it
 contains a uniformly discrete set  $\Lambda '$ with separation constant $\delta $, say. We
 now choose $\gamma >0$ so small that
 \begin{equation}
   \label{eq:26}
 1- \Big( 1-\frac{(1+\delta )\gamma}{n} \Big)^{2n+2} \leq
 \frac{A \delta ^2}{8 } \, \qquad \text{ for all } n > 2\gamma \,  .  
 \end{equation}
 This is possible   because 
 $\Big( 1-\frac{(1+\delta )\gamma}{n} \Big)^{2n+2} \geq e^{-4(1+\delta
   )\gamma } ( 1-\frac{(1+\delta )\gamma}{n} )^{2}$  by the lower
 estimate in \eqref{eq:10},  and thus  tends to $1$ uniformly in $n$, as
 $\gamma \to 0$.    Note that the choice of $\gamma $ depends  on the
 lower sampling constant $A$ and  the separation
  of $\Lambda ' $.
 
We now apply Lemma~\ref{boundary} and Corollary~\ref{negl} to the points $\Lambda \setminus \Lambda _n
= \Lambda \cap  \cro $. Set $\gamma ' = (1+\delta )\gamma $.  Then for $p \in \poln $ we obtain
\begin{align*}
  \sum _{\lambda \in \Lambda \setminus \Lambda _n} \frac{|p(\lambda
  )|^2}{k(\lambda , \lambda )}  & \leq \frac{4}{\pi \delta ^2} \int _{C_{\gamma '/n}}
                                  |p(w)|^2 \, dw \\
  &\leq \frac{4}{\delta ^2} \Big(  1- \Big( 1-\frac{(1+\delta )\gamma}{n}
    \Big)^{2n+2} \Big) \|p\|_{A^2}^2 \leq \frac{A}{2} \|p\| \ano  ^2
    \, .
\end{align*}
Consequently,
\begin{align*}
    \sum _{\lambda \in \Lambda _n} \frac{|p(\lambda
  )|^2}{k(\lambda , \lambda )}  &= \sum _{\lambda \in \Lambda } - \sum
                                  _{\lambda \in \Lambda \setminus
                                  \Lambda _n} \dots \\
  &\geq A \|p\| \ano ^2 -   \sum _{\lambda \in \Lambda \setminus \Lambda _n} \frac{|p(\lambda
  )|^2}{k(\lambda , \lambda )}  \geq (A - A/2) \|p\| \ano ^2 \, .
\end{align*}
Since always $k_n(\lambda ,\lambda ) \leq k(\lambda , \lambda )$ we
finish the estimate by
$$
    \sum _{\lambda \in \Lambda _n} \frac{|p(\lambda
  )|^2}{k_n(\lambda , \lambda )} \geq     \sum _{\lambda \in \Lambda _n} \frac{|p(\lambda
  )|^2}{k(\lambda , \lambda )} \geq \frac{A}{2} \|p\| \ano ^2 \qquad
\text{ for all } p \in \poln \, .
$$
As usual, the upper bound is easy. Since by \eqref{eq:5} $k_n(\lambda , \lambda )
\asymp k(\lambda , \lambda )$ for $\lambda \in \br $, it
follows that
\begin{equation*}
        \sum _{\lambda \in \Lambda _n} \frac{|p(\lambda
  )|^2}{k_n(\lambda , \lambda )} \lesssim _\gamma     \sum _{\lambda \in \Lambda _n} \frac{|p(\lambda
                                   )|^2}{k(\lambda , \lambda )} 
  \leq     \sum _{\lambda \in \Lambda } \frac{|p(\lambda
  )|^2}{k(\lambda , \lambda )} \leq  B \, \|p\|
    \ano ^2 \, ,
\end{equation*}
since $\Lambda $ is sampling for $A^2(\bD )$. 
\end{proof}

Since sampling sets for $A^2(\bD )$ are completely characterized by
means of the Seip-Korenblum density~\cite{seip93}, Theorem~\ref{samptomz} provides a
wealth of examples of \mz\ families. However, 
the parameter $\gamma $ in
the statement depends on the lower sampling constant $A$ and the
separation $\delta $ of $\Lambda $. 


\subsection{From \mz\ inequalities to sampling} \label{sec:mzsamp}
Choose either  the Euclidean or the pseudohyperbolic metric on $\bD $
and let $d(E,F)$ be
the corresponding Hausdorff distance between two closed sets $E,F\subseteq \bD $. We
say that  a sequence of sets $\Lambda _n\subseteq \bD $ converges weakly to  
$\Lambda \subseteq \bD  $, if  for all compact disks  $B\subseteq \bD $
$$
\lim _{n\to \infty } d\big((\Lambda _n \cap B) \cup \partial B, (\Lambda  \cap B) \cup
\partial B \big) = 0 \, . 
$$
See~\cite{DS04,GOR15,HKZ00} for equivalent definitions. The main consequence of weak
convergence  is the convergence of sampling sums. If all $\Lambda _n$
are uniformly separated with fixed separation $\delta $,  then 
$$
\sum _{\lambda \in \Lambda _n \cap B}
\frac{|f(\lambda)|^2}{k(\lambda,\lambda )} \to \sum _{\lambda \in \Lambda  \cap B}
\frac{|f(\lambda)|^2}{k(\lambda,\lambda )} \, .
$$
for all $f\in A^2(\bD )$. More generally, if each $\Lambda _n$ is a
finite union of $K$ uniformly separated sets with separation constant
$\delta $ independent of $n$, then we need to include multiplicities
$m (\lambda ) \in \{1, \dots , K\}$. We then  obtain 
\begin{equation}
  \label{eq:fin1}
\sum _{\lambda \in \Lambda _n \cap B}
\frac{|f(\lambda)|^2}{k(\lambda,\lambda )} \to \sum _{\lambda \in \Lambda  \cap B}
\frac{|f(\lambda)|^2}{k(\lambda,\lambda )} m(\lambda )\, .
\end{equation}
Since the multiplicities $m (\lambda )$  are bounded, they only affect the
constants, but do not change the arguments.

\begin{tm} \label{mzsampa}
  Assume that
  $(\Lambda _n)$ is a \mz\ family for the polynomials $\poln $ in
  $A^2(\bD )$. Let $\Lambda $ be a weak limit of $(\Lambda
  _n)$ or some subsequence $(\Lambda _{n_k})$.
  Then $\Lambda $ is a sampling set for $A^2(\bD )$. 
\end{tm}
\color{black}
\begin{proof}
\textbf{Step 1.} We assume that $\Lambda_n$ is a Marcinkiewicz Zygmund family  and
therefore there are  $A,B> 0$ such that 
$A\|p\| \ano\leq \sum _{\lambda \in \Lambda _n} \frac{|p(\lambda
    )|^2}{k_n(\lambda,\lambda)}\leq B\|p\|^2\ano $ for all polynomials
  $p\in \cP _n$.
The upper inequality implies that   there are $C>0$
and $\gamma> 0$ such that $\#(\Lambda_n \cap B(w, 0.5(1-|w|))) \le C$
for all $w\in \br $ and all $n$ (Theorem~\ref{besselgeom}). 
In addition, each   set $\Lambda_n \cap \br $ is  a finite
  union of $K$ uniformly separated sequences with separation $\delta $
  independent of $n$.

  Since $\Lambda _n$ converges to $\Lambda $ weakly,  $\Lambda$ satisfies the same inequality 
  \[
   \#(\Lambda \cap \bar{B}(w, 0.5(1-|w|))) \le C,\qquad \forall w\in \br \, ,
  \]
and thus  $\Lambda$ is also  a union of $K$ uniformly discrete
  sequences with separation $\delta $, see, e.g.,~\cite{DS04}, Sec.~2.1, Lemma 16.

    Note that the geometric conditions furnish constants $c_\gamma $ for
  the comparison of the kernels $k$ and $k_n$ on $\br$, the
  separation constant $\delta $, and the  multiplicity $K$.  In the
  next step we will choose a suitable radius $r$ that solely depends
  on these given constants. 

   \textbf{Step 2.} 
   We prove the desired sampling inequality in $A^2(\bD )$ for all polynomials and
  then use a density argument. 

Fix a polynomial $p$ of degree $N$, say.  We know that there is an $r = r(N)<1$ such that
$\int_{|z|> r-\delta } |p|^2 \le \frac{c_\gamma \delta ^2 A }{8K}
\|p\|^2\ano$. We take $n\geq N$ big enough such that $r< 1 - \gamma /n$. In this case we have
\begin{align*}
A \|p\|^2 \ano \le  \sum _{\lambda \in \Lambda _n} \frac{|p(\lambda
    )|^2}{k_n(\lambda,\lambda)} &= 
    \sum _{\lambda \in \Lambda _n, |\lambda| <r} \frac{|p(\lambda
                 )|^2}{k_n(\lambda,\lambda)} \\
  &+ \sum _{\lambda \in \Lambda _n, r\leq
      |\lambda| <\gn } \frac{|p(\lambda
    )|^2}{k_n(\lambda,\lambda)} + \sum _{\lambda \in \Lambda _n,
      |\lambda | \geq \gn  } \frac{|p(\lambda
    )|^2}{k_n(\lambda,\lambda)} = \\
   & = I_n + II_n + III_n.
\end{align*}
According to Lemma~\ref{wei1}  we may replace $k_n(\lambda, \lambda)$
by $k(\lambda, \lambda)$ in the first two terms 
and  by $n^2$ in the third term. 
Thus with  a constant depending on $\gamma $, we obtain 
\begin{align} \label{eq:fin2}
A \|p\|^2 \ano \leq c_\gamma \inv \Big( \sum _{\lambda \in \Lambda _n \cap B_r} \frac{|p(\lambda
    )|^2}{k(\lambda,\lambda)} +  \sum _{\lambda \in \Lambda _n, r\leq |\lambda| <\gn } \frac{|p(\lambda
    )|^2}{k(\lambda,\lambda)} +  \sum _{\lambda \in
  \Lambda _n \cap C_{\gamma/n} } \frac{|p(\lambda
    )|^2}{n^2} \Big) \, .
\end{align}
In the sum  $I_n$  all points $\lambda$ lie in the compact set
$\overline{B(0,r)}$, and  by weak convergence  including
multiplicities $m(\lambda ) \in \{1, \dots, K\}$,  we obtain 
\begin{align*}
\lim _{n\to \infty } I_n &\leq c_\gamma \inv  \lim_{n\to\infty}\sum_{\lambda \in \Lambda _n \cap B_r} \frac{|p(\lambda
    )|^2}{k(\lambda,\lambda)} = \\
    & = c_\gamma \inv 
     \sum_{\lambda \in \Lambda \cap \overline{B_r}} \frac{|p(\lambda
    )|^2}{k(\lambda,\lambda)} m(\lambda) \leq c_\gamma \inv K
  \sum_{\lambda \in \Lambda \cap \overline{B_r}} \frac{|p(\lambda 
    )|^2}{k(\lambda,\lambda)} \, .
\end{align*}
To treat $II_n$, we  use the assertion of Step~1 that every  $\Lambda_n \cap
\br $ is  a finite union of at most $K$ uniformly separated
sequences with separation $\delta $. Lemma~\ref{boundary} and the
choice of $r$  yield
\[
II_n \leq c_\gamma \inv  \sum _{\lambda \in \Lambda _n, r\leq |\lambda| <1-\gamma/n } \frac{|p(\lambda
    )|^2}{k(\lambda,\lambda)} \le \frac{4K}{\pi \delta ^2c_\gamma } \int_{|z|>
    r-\delta } |p(z)|^2 \le  \frac{A}{2} \|p\|\ano.
\]
Finally, the last term  $III_n$ is negligible when $n\to\infty$ because 
\begin{align*}
III_n =   \sum _{\lambda \in \Lambda_n \cap C_{\gamma/n}}  \frac{|p(\lambda)|^2}{n^2} \le \|p\|_\infty ^2  \, \frac{1}{n^2} 
\# \{ \lambda \in \Lambda _n \cap C_{\gamma/n} \} \, . 
\end{align*}
By Lemma~\ref{boundarycount}(i) $III_n$ tends to $0$, as $n\to \infty
$.

Finally we take the limit in \eqref{eq:fin2} and  
obtain  
inequality
  \begin{equation}
    \label{eq:16}
 A \|p\| \ano ^2 \leq c_\gamma \inv K   \sum _{\lambda \in \Lambda
   \cap \overline{B(0,r)} } \frac{|p(\lambda
    )|^2}{k(\lambda,\lambda )}  +  \frac{A}{2} \|p\|^2 \ano  \, .
\end{equation}
 This implies  the lower sampling inequality valid for all polynomials. 

Since the polynomials are dense in $A^2(\bD
)$, \eqref{eq:16} can be extended to all of $A^2(\bD )$.~\footnote{Note
that $\sum _{\lambda \in \Lambda } \frac{|p(\lambda
    )|^2}{k(\lambda,\lambda )} = \sum _{\lambda \in \Lambda } |\langle
  p,  \kappa_\lambda \rangle |^2$, where $\kappa_\lambda (z) = k(z,\lambda)
  k(\lambda,\lambda )^{-1/2}$ is the normalized reproducing
  kernel. Thus the left-hand side  in \eqref{eq:16} is just
  the frame operator associated to the set $\{\kappa_\lambda \}$. For
  boundedness and invertibility it therefore suffices to check on a
  dense subset.}

\textbf{Step 3.} As always, the upper bound is much easier to prove: Let $p\in \cP
_N$. Then by Lemma~\ref{boundary}
$$
\sum _{\lambda \in \Lambda \cap C_{\gamma /N}} \frac{|p(\lambda
  )|^2}{k(\lambda,\lambda)} \leq C \int _{C_{\gamma'/N}} |p(w)|^2 \,
dw \leq C \|p\|^2 \ano \, .
$$
On the disk $B_{1-\gamma /N}$ we use the weak convergence and deduce
that 
$$
\sum _{\lambda \in \Lambda \cap B_{1-\gamma /N}} \frac{|p(\lambda
  )|^2}{k(\lambda,\lambda)} = \lim _{n\to \infty } \sum _{\lambda
  \in \Lambda _n \cap \bar{B}_{1-\gamma /N}} \frac{|p(\lambda
  )|^2}{k_n(\lambda,\lambda)} \leq B \|p\|^2 \ano \, ,
$$
because $\Lambda _n$ is a \mz\ family for polynomials and $p\in \cP _N
\subseteq \poln $. 
The sums  of both terms  yields the upper sampling inequality
for $\Lambda $ for all polynomials, which extends to $A^2(\bD )$ by
density.  
\end{proof}

The upper bound  can also be derived  from the geometric
  description  of Theorem~\ref{besselgeom}. 

  \vspace{3mm}
  

With a bit more effort one can prove an $A^p$-version of
  Theorems~\ref{samptomz} and \ref{mzsampa}. The proof requires
  several modifications of interest. For instance, to obtain \eqref{eq:9}
  for the $A^p$-norm, one needs an argument similar to ~\cite[Lemma
  7]{OS07}. The $A^p$-norm of the normalized reproducing kernel in
  Proposition~\ref{boundarycount} requires the boundedness of the Bergman
  projection on $L^p$. 


\section{Hardy space}

\subsection{Basic facts.}
The Hardy space $H^2= H^2(\bD)$ consists of all analytic functions in
$\bD $ whose boundary values on $\partial \bD = \bT $ are in $L^2(\bT )$ with  finite norm
\begin{equation}
  \label{eq:1a}
  \|f\|_{H^2} =  \Big(\int _{0 }^1 |f(e^{2\pi i t})|^2 \, dt \Big)^{1/2} \, .
\end{equation}
 The monomials $z\to z^k$ form an  orthonormal basis,  and the norm of $f(z) =
\sum _{k=0}^\infty a_k z^k$ is 
\begin{equation}
  \label{eq:2a}
  \|f\|_{H^2} ^2 = \sum _{k=0}^\infty |a_k|^2  \, .
\end{equation}
For  $p(z) = \sum _{k=0}^n a_k z^k \in \poln $ and $0< \rho \leq 1$, let
$p_\rho(z) = p(\rho z)$,  then 
\begin{align*}
\|p_\rho \|^2\hano &= \sum _{k=0}^n |a_k|^2 \rho ^{2k}        
\geq  \rho ^{2n} \|p\|^2 \hano  \, , 
\end{align*}
and clearly $\|p_\rho \| \hano \leq \|p\| \hano $. 
If $p\in \poln $,  $\rho = \gn $, and $n>2\gamma $,  then
by   \eqref{eq:10}
\begin{equation}
  \label{eq:23}
\|p\|^2 \hano \geq \|p_\rho \|^2 \hano \geq (\gn )^{2n} \|p\|^2 \hano
\geq e^{-4\gamma } \|p\|^2\hano \, .
\end{equation}

The reproducing kernel of $\cP _n$ in $H^2$ is given by
\begin{align}  
  k_n(z,w) & = \sum _{k=0}^n (z\bar{w})^n =   \frac{1 - (z\bar{w})^{n+1}}{1-z\bar{w}} \, .  \label{eq:3a}
\end{align}
As $n\to \infty $, the kernel tends to the reproducing kernel of
$H^2$,
\begin{equation}
  \label{eq:4}
k(z,w) =   \frac{1}{1-z\bar{w}}
\end{equation}
for $z,w \in \bD $. 
For \mz\ families for polynomials in  Hardy space it will be necessary to also consider sampling points
\emph{outside} the unit disk as in ~\cite{OS07}. With this caveat in mind, we define the
appropriate annuli as
\begin{equation}
  \label{eq:17}
  A_{\gamma/n} = \{ z\in \bC : \gn \leq |z| \leq (\gn )\inv \} \, .
\end{equation}
We now  compare the kernels for $\poln $ and $H^2$. 
\begin{lemma}
  \label{xei1}
  Let $k_n(z,w)$ be the reproducing kernel of $\poln$ in $H^2$ and
  let $\gamma >0$ be arbitrary. 


   If $z\in \ar $, i.e., $\gn \leq |z| \leq (\gn )\inv$,  and $n>2\gamma $, then
  \begin{equation}
    \label{eq:6a}
    \frac{1-e^{-4\gamma } }{2\gamma} n \leq k_n(z,z) \leq
    \frac{e^{4\gamma }}{\gamma} n  \, .
  \end{equation}
  Consequently, if $\Lambda _n \subseteq \ar $, then,   for all polynomials $p\in \poln $,
  \begin{equation}
    \label{eq:18}
       \frac{1}{n} \sum _{\lambda \in \Lambda_n} |p(\lambda )|^2
  \asymp \sum _{\lambda \in \Lambda_n} \frac{|p(\lambda
    )|^2}{k_n(\lambda,\lambda)}    \, .
  \end{equation}
\end{lemma}
On the annulus $\ar $ we may therefore always work with the weight
$1/n$ for polynomials of degree $n$. 
\begin{proof}
Writing $r=|z|$, the kernel  $k_n(z,z) = \frac{1-r^{2n+2}}{1-r^2}$
is increasing in $r$, so that for $z\in \ar $ we have
$$
\frac{1-(\gn )^{2n+2}}{1- (\gn )^2} \leq k_n(z,z) \leq \frac{(\gn
  )^{-2n-2} - 1}{(\gn )^{-2} -1} = \frac{(\gn
  )^{-2n} - (\gn )^2 }{1  - (\gn )^{2} } \, .
$$
For $n\geq 2\gamma $, the denominator $1  - (\gn )^{2} = \frac{\gamma
}{n} + \frac{\gamma }{n}(1- \frac{\gamma }{n})$ is  between $\gamma /n$ and $2\gamma /n$, for the
numerator we use \eqref{eq:10} and obtain $(\gn
  )^{-2n} - (\gn )^2  \leq e^{4\gamma }$ and  $1-(\gn )^{2n+2} \geq
  1-(\gn )^{2n} \geq 1-e^{-2\gamma }$. Thus both ratios are of order
  $n$ with the  constants $(1-e^{-2\gamma }) /(2\gamma)$ and
 $ e^{4\gamma}/\gamma $. 
 \end{proof}

\subsection{From \mz\ families on \texorpdfstring{$\bT $}{T} to \mz\ families for
  \texorpdfstring{$H^2(\bD)$}{H2(D)}. }

In contrast to the situation for the Bergmann space, there are no
sampling sequences for the Hardy space $H^2(\bD )$ by the results of
Thomas~\cite{thomas98}. Duren and Schuster~\cite[p.~154]{DS04} give
the following simple argument:  a sampling set $\Lambda $ for $H^2$
must be a Blaschke sequence. However, every Blaschke sequence
is a zero set in $H^2$, which contradicts the sampling inequality.   Therefore there can be no analogue of
Theorem~\ref{samptomz}.

Nevertheless, we prove that Hardy space admits \mz\
families for polynomials. The idea is to associate a \mz\ for
polynomials in $H^2(\bD )$ to every  \mz\ family  for polynomials on
the torus. As these are well understood~\cite{Gro99,OS07}, 
we obtain a general class of \mz\ families in $H^2(\bD )$.

 We use the following notation: for $\lambda \in \bC
\setminus \{0\}$ let $\tilde{\lambda}  = \frac{\lambda}{|\lambda|}$ be
the projection from the complex plane $\bC\setminus \{0\}$ onto the
torus $\partial \bD = \bT$.

\begin{tm}
  \label{mztorushardy}
  Assume that the family $(\Lambda _n) \subseteq \bC $ has the
  following properties:

  (i) there exists $\gamma >0$, such that $\Lambda _n \subseteq \ar $
  for all $n \geq \gamma $.

  (ii) The projected family $(\tilde{\Lambda}_n) \subseteq \bT $ is a
  \mz\ family for the polynomials $\poln \subseteq L^2(\bT )$.

    (iii) The projection $\Lambda _n \to \widetilde{\Lambda_n}$ from
  $\ar $ to $\bT $ is one-to-one for all $n$.
  

  Then $(\Lambda _n)$ is a \mz\ family for the polynomials $\poln $ in
  $H^2(\bD )$. 
\end{tm}

The statement in the introduction is just a reformulation of Theorem~\ref{mztorushardy} without
additional  notation. 

The proof is inspired by a sampling theorem of Duffin and Schaeffer
for bandlimited functions from samples in the complex plane (rather
than from samples on the real axis). See~\cite{DS52} and
\cite{seip04,young80}. In analogy to the theory of bandlimited
functions, one can view Theorem~\ref{mztorushardy} also as a
perturbation result 
for \mz\ families in $L^2(\bT )$, where the points in $\bT $ are
perturbed in a complex neighborhood of $\bT$. 

We wrap the  main part of the proof into a technical
lemma.

\begin{lemma}
  \label{teclemma}
Let $\gamma >0$ and $n> 2\gamma $.   Assume that $\Lambda _n $ is a
finite set contained in $\ar $ so that the projection
$\Lambda _n \to \widetilde{\Lambda_n}$ is one-to-one. 
Then 
  there exists a set $\Lambda _n ^{(1)}$ with the
  following properties:

  (i) $\Lambda _n^{(1)}$ is   contained in the smaller annulus
  $A_{\frac{3\gamma}{4n}}$ and $\# \Lambda _n^{(1)} = \# \Lambda _n$.

 (ii)  For every $\lambda \in \Lambda _n$ there is a $\mu \in \Lambda
  _n ^{(1)}$, such that $\tilde{\lambda} = \tilde{ \mu }$, and

  (iii) For every $p\in \poln $ there exists $p_1\in \poln $ satisfying
  $$\|p_1\|^2 \hano \geq e^{-2\gamma/3} \|p\| \hano ^2 $$ and
  \begin{equation}
    \label{eq:19}
 \frac{1}{n} \, \sum _{\lambda \in \Lambda
      _n} |p(\lambda )|^2 \geq   \frac{1}{n} \, \sum _{\lambda \in \Lambda
      _n^{(1)}}|p_1(\lambda )|^2 \, . 
  \end{equation}
\end{lemma}
As we will apply the lemma to a \mz\ family, it is important the
constants and the construction are independent of the degree $n$.
\begin{proof}
  \textbf{Step 1.} Construction of $\Lambda _n^{(1)}$. If $\lambda \in
  \Lambda _n$ and $|\lambda |\leq 1$, set $\mu = (1+\frac{\gamma}{3n})
  \lambda $. If $\lambda \in
  \Lambda _n$ and $|\lambda |> 1$, set $\mu = (1+\frac{\gamma}{3n})/ \bar{\lambda} $.  By construction, $\tilde{\mu} =
  \tilde{\lambda}$. Since $\Lambda _n \to \widetilde{\Lambda _n}$ is
  one-to-one  by   assumption,   $\Lambda _n^{(1)}$ has the same
  cardinality as $\Lambda _n$. 
  [To appreciate this
  assumption, consider the case when both $\lambda $ and
  $1/\bar{\lambda}$ are in $\Lambda _n$. They both would be mapped to
  the same point $\mu $.] Since for   $|z |\geq \gn $, we have  $
  (1+\frac{\gamma}{3n}) |z | \geq  (1+\frac{\gamma}{3n}) (\gn
  ) \geq 1- \frac{3\gamma}{4n}$, and for  $ |z|<1$ we have
  $(1+\frac{\gamma}{3n}) \leq (1-\frac{3\gamma}{4n})\inv $, it
  follows that  $\Lambda
  _n^{(1)}$ is contained in the smaller annulus $A_{3\gamma/(4n)}$. 

  \textbf{Step 2.} Construction of $p_1$. Given $p\in \poln $ with
  zeros $z_j$ and factorization $p(z) = z^\ell \prod  (z-z_j)$, we 
  obtain $p_1$ 
  by reflecting all its zeros into $\bD $ and an   appropriate
  scaling. We multiply $p$ by several Blaschke factors and set
 \begin{equation}
   \label{eq:20}
   \widetilde{p_1}(z) = z^l \prod _{|z_j|\leq 1} (z-z_j) \prod _{|z_j|> 1}
   (1 - \overline{z_j}z) = p(z) \,  \prod _{|z_j|> 1}
   \frac{1 - \overline{z_j}z}{z-z_j} \, .  
 \end{equation}
By construction, all zeros of $\widetilde{p_1}$ are now in
the unit disk $\overline{\bD}$. In engineering terminology,
$\widetilde{p_1}$ is the minimum phase filter associated to $p$.  Furthermore,  $\widetilde{p_1}$ has the
  following properties:

  (i) By \eqref{eq:20} and the property of Blaschke factors we have 
$$
|\widetilde{p_1}(z)| = |p(z)| \qquad \text{ for } z\in \bT \, , 
$$
and thus $\|\widetilde{p_1}\| \hano = \|p\|\hano $. 

(ii) For $ z\in \bD $ we have
\begin{equation}
  \label{eq:21}
  |\widetilde{p_1}(z)| \leq \min \Big( |p(z)|, |p(1/\bar{z})| \Big) \, .
\end{equation}
To see this, observe that for $|z_j| > 1$  each Blaschke
factor in \eqref{eq:20} satisfies 
$$
\Big| \frac{1 -\bar{z}_j z}{z-z_j} \Big| = \Big| \frac{
    \bar{z}_j\inv - z}{z_j\inv z -1}\Big| \leq 1 \, ,
$$
thus $|\widetilde{p_1}(z)| \leq |p(z)|$ for $z\in \bD $. 
Using the first factorization of $\widetilde{p_1}$ in \eqref{eq:20} and
$p(1/\bar{z}) = \bar{z}^{-l} \prod _{|z_j| \leq 1} \big( 1/\bar{z} -
z_j \big) \,  \prod _{|z_j| > 1} \big( 1/\bar{z} -
z_j \big) $, the second inequality in \eqref{eq:21} follows from 
$$
\frac{|\widetilde{p_1}(z)|}{|p(\frac{1}{\bar{z}})|} = |z|^{2l} \prod
_{|z_j|\leq 1} \frac{|z-z_j|}{|\bar{z}\inv (1-\bar{z}z_j)|} \, 
\prod _{|z_j|> 1} \frac{|1- \bar{z_j}z|}{|\bar{z}\inv
    (1-\bar{z}z_j)|} \leq 1 \, .
$$
Finally we set 
\begin{equation}
  \label{eq:22}
  p_1(z) = \widetilde{p_1}\Big( (1+ \frac{\gamma}{3n})\inv z \Big) \in
  \poln \, .
\end{equation}
Then by \eqref{eq:23}
$$
\|p_1\|^2 \hano \geq (1+ \frac{\gamma}{3n})^{-2n} \|\widetilde{p_1} \|^2 \hano
\geq e^{-2\gamma /3} \|p\|^2 \hano \, .
$$
and obviously $\|p_1\|^2 \hano \leq \|\widetilde{p_1}\|\hano ^2 =  \|p\|^2\hano $. 

\textbf{Step 3.} Sampling on $\Lambda _n ^{(1)}$. If $\mu = (1+
\frac{\gamma}{3n}) \lambda \in \Lambda _n ^{(1)} \subseteq \bD $, then
by \eqref{eq:21}
$$
|p_1(\mu )| = |\widetilde{p_1} (\lambda )| \leq |p(\lambda )| \, ;
$$
if $\mu = (1+
\frac{\gamma}{3n}) / \bar{\lambda }\in \Lambda _n ^{(1)} \subseteq \bD $, then 
$$
|p_1(\mu )| = |\widetilde{p_1} (1/ \bar{\lambda } )| \leq |p(\lambda )| \, .
$$
In conclusion, we obtain
$$
 \frac{1}{n} \, \sum _{\mu \in \Lambda
      _n^{(1)}} |p_1(\mu )|^2 \leq   \frac{1}{n} \, \sum _{\lambda \in \Lambda
      _n} |p(\lambda )|^2 \, ,
  $$
which was to be shown. 
\end{proof}

\begin{proof}[Proof of Theorem~\ref{mztorushardy}]
Applying Lemma~\ref{teclemma} repeatedly, 
after  $\ell$ steps we  construct  a set $\Lambda _n^{(\ell )}  \subseteq
\bD $ with the following properties:
\begin{align}
  \label{eq:24}
  \Lambda _n ^{(\ell )} & \subseteq A_{\gamma _\ell / n}  \qquad
  \text{ with }  \gamma _\ell = \tfrac{3}{4} \gamma _{\ell -1} \, , \\
  \widetilde{\Lambda _n ^{(\ell )}} &= \widetilde{\Lambda _n} \, . \label{eq:30}
 \end{align}
 Furthermore, for given $p \in \poln $ we construct a sequence on
 polynomials $p_1, \ldots , p_{\ell} , \ldots$ with decreasing norm
 $\|p_\ell \| \hano \leq \dots \leq \|p_1\| \hano \leq \|p\| \hano $,  such that
 \begin{equation}
   \label{eq:25}
\| p_\ell \| ^2 \hano \geq e^{-\tfrac{2}{3} \gamma _{\ell -1} }\|p_{\ell -1} \|
  ^2 \hano \geq e^{-\tfrac{2}{3} (\gamma _{\ell -1} + \gamma _{\ell
      -2} )} \|p_{\ell -2} \|^2 \hano \geq \dots \geq e^{-\tfrac{2}{3}
    \sum _{j=0}^{\ell -1} \gamma _{j}} \|p \|^2 \hano  \, .
\end{equation}
Since $\gamma _j = 3\gamma _{j-1}/4$ and $\gamma _0 = \gamma $, we
find $\gamma _{\ell } = \big(\frac{3}{4}\big)^{\ell }\gamma$ and 
$$\tfrac{2}{3}
    \sum _{j=0}^{\ell -1} \gamma _{j} \leq  \tfrac{2}{3} \sum
    _{j=0}^{\infty } (3/4)^j \gamma \leq \tfrac{8}{3}\gamma \, .
    $$
    It follows that always $\|p_\ell \|^2\hano \geq
    e^{-\tfrac{8}{3}\gamma } \|p\|^2\hano $.
    
We now let $\ell $ tend to $\infty $. Since the sequence $p_\ell $ is
bounded in the finite-dimensional space $\poln $, it contains a convergent subsequence such
that $\lim _{j\to\infty } p_{\ell _j} = p_\infty \in \poln
$. Furthermore, by \eqref{eq:25} 
the limiting polynomial $p_\infty $ must be non-zero. By \eqref{eq:24}
every point in $\Lambda _n ^{(\ell )} $ 
    converges to a  point on the torus, precisely to the corresponding
    point in the projection
   $\widetilde{\Lambda _n}$. Using \eqref{eq:25}, it follows that
$$
    \frac{ \frac{1}{n} \, \sum _{\lambda \in \Lambda
      _n} |p(\lambda )|^2}{\|p\|^2 \hano }
   \geq e^{-8\gamma/3 }   \, \frac{ \frac{1}{n} \, \sum _{\lambda \in \widetilde{\Lambda
      _n}} |p_\infty (\lambda )|^2}{\|p_\infty \|^2\hano } 
  \, .   
$$
Finally, we recall the assumption that the projected family $\widetilde{\Lambda _n}$ is
a \mz\ family for the  polynomials of degree $n$  in
$L^2(\bT )$. 
Consequently, we obtain that 
$\frac{1}{n} \, \sum _{\lambda \in \widetilde{\Lambda
      _n}} |p_\infty (\lambda )|^2 \geq A \|p_\infty \|^2\hano $, 
  which implies the corresponding sampling inequality for $p\in \cP _{n}$.

 The upper bound is proved almost exactly as the corresponding\footnote{Note that
    \cite{OS07} uses the weights $m_n=\# \Lambda _n$ instead of
    $k_n(\lambda,\lambda) \asymp n$. This does not affect the
    estimates.}
 statement for $\poln $ in $L^2(\bT )$ in Thm.~9 of \cite{OS07}. Since
  $\widetilde\Lambda_n$ is a \mz\ family  for $\poln $ in $L^2(\bT )$
  by our assumption, \cite[Thm.~9]{OS07} asserts that for every
  interval $I\subseteq \bT $  of length $1/n$ we have   $\#
  (\widetilde{\Lambda _n} \cap I) \leq C$.  Since $\Lambda _n \subseteq \ar $, this condition
  implies that $\# \Lambda _n \cap B(z,1/n) \leq C'$ for all $z\in
  \ar $. This geometric condition now yields the upper bound
  $\tfrac{1}{n} \sum _{\lambda \in \Lambda _n} |p(\lambda)|^2 \lesssim
    \|p\| \hano ^2$ for all $p\in \poln $ precisely as in \cite{OS07}. Indeed
    that proof uses the submean-value property and the extension to
    $\ar $.   
\end{proof}

\vspace{3mm}

Theorem~\ref{mztorushardy} shows that to every \mz\ family for
 polynomials on $\bT $   we can associate \mz\ families in $H^2(\bD )$ by moving points from the boundary $\bT =
\partial \bD $ into a carefully controlled annulus $\cro \subseteq \bD $. The
following example investigates the role of points in the interior of
$\bD $ for \mz\ families. 

\vspace{3mm}

\noindent \textbf{Example.} We   construct an example of a \mz\
family $(\Lambda _n)$ for polynomials $\poln $ in $H^2(\bD )$, so that
  $(\Lambda _n \cap \cro )$ is not a \mz\ family in
$H^2$ and 
the projection  $(\widetilde{\Lambda _n})$ is not  a \mz\ family for
$\poln $ in $L^2(\bT)$. 
Let $\gamma >0$, $\alpha_n >0$, and 
\begin{equation}
  \label{eq:27}
  \Lambda _n = \{ (\gn ) e^{2\pi i k/n} : k=0, \dots , n-1 \} \cup \{
  \alpha _n e^{2\pi i /n^2} \} \, .
\end{equation}

(i) If  $\alpha _n < \gn $, then $\# (\Lambda _n \cap \cro )= n < \mathrm{dim}
\, \poln $ and thus $(\Lambda _n\cap \cro  )$ cannot be  a \mz\ family
for $\poln $ in $H^2$. 

(ii) The projected family  $\widetilde{\Lambda _n} = \{  e^{2\pi i k/n} : k=0, \dots , n-1 \} \cup \{e^{2\pi i /n^2}\}$
is not a \mz\ family for $\poln $ in
$L^2(\bT )$. We 
  choose $p(z) = z^n -1$ with $\|p\|^2 \hano =2$. Then $p(e^{2\pi i k/n}) = 0$ and $|p(e^{2\pi
    i /n^2})| = |e^{2\pi i /n} -1|  \lesssim 1/n $, so that
$$
\frac{1}{n+1} \sum _{\lambda \in \widetilde{\Lambda _n}} |p(\lambda )|^2
\lesssim \frac{1}{n^3} \, ,
$$
violating the sampling inequality for large $n$. 

(iii) However, $(\Lambda _n)$ is a \mz\ family for $\poln $ in
$H^2(\bD )$. To see this, we  consider the modified set $\Lambda
_n' = \{ (\gn ) e^{2\pi i k/n} : k=0, \dots , n-1 \} \cup \{ 0 \}$ and
then use a perturbation argument. 

We write  $p\in \poln $ as  $p(z) = p(0) + z \widetilde p(z)$ 
for a unique $\widetilde p \in \cP _{n-1}$. Then  $\|p\| ^2 = |p(0)|^2 +
\|\widetilde p\|^2$. Let  
$q(z) = \widetilde
p((1-\frac{\gamma}{n}) z)$,  then by ~\eqref{eq:23} $\|q\|\hano ^2 \geq
e^{-4\gamma } \|\widetilde p\|\hano ^2$. Calculating the norm of $q$
by sampling, we obtain 
\begin{align*}
\|\widetilde p\|^2 \asymp \|q\|^2 &= \frac 1{n} \sum_{j= 0}^{n-1}|q(e^{2\pi i j/n})|^2 =
\frac 1n \sum_{j= 0}^{n-1}|\widetilde p\big((1-\frac{\gamma}{n})e^{2\pi i j/n}\big)|^2  \\
&= \frac {1}{n(1-\frac{\gamma}{n})^2} \sum_{j=
                                                                                     0}^{n-1}|(1-\frac{\gamma}{n}) e^{2\pi i j/n}\widetilde p((1-\frac{\gamma}{n})e^{2\pi i j/n})|^2 \, .
\end{align*}
If  $n\ge 2\gamma$ we have
\[
 \|q\|^2 \le \frac 8{n} \sum_{j= 0}^{n-1}\Big(|p((1-\frac{\gamma}{n})e^{2\pi i j/n})|^2 + |p(0)|^2\Big)
\]
Since by \eqref{eq:6a} $k_n(\lambda ,\lambda ) \asymp n$ for $\lambda = (\gn
) e^{2\pi i j/n}$, and $k_n(0,0) = 1$, the above inequality states
that 
\[
 \|q\|^2 \lesssim \sum_{\lambda\in \Lambda_n'} \frac{|p(\lambda)|^2}{k_n(\lambda, \lambda)}.
\]
and finally
\[
 \|p\|^2 \le \|q\|^2 + |p(0)|^2 \lesssim \sum_{\lambda\in \Lambda_n'} \frac{|p(\lambda)|^2}{k_n(\lambda, \lambda)}.
\]
Thus $(\Lambda _n')$ is a \mz\ family for $\poln $ in $H^2(\bD )$. 
    The small perturbation $0\to \frac 1{n^2} e^{2\pi i /n^2}$ is of
    order $1/n^2$ and thus  preserves the sampling inequality

\vspace{3mm}

As already mentioned,  the standard definition of sampling sequences is
vacuous in the Hardy space. Thus Thomas in \cite{thomas98} proposed an alternative
definition of sampling in terms of the non-tangential
maximal function $M_\Lambda$. 
\[
 M_\Lambda(f) (e^{i\theta}) := \sup_{\Gamma(e^{i\theta})\cap \Lambda} |f|,
\]
where $\Gamma (e^{i\theta}) =
\{ z\in \bD : \tfrac{|z-e^{i\theta }|}{1-|z|} < 1+ \alpha\}$ is a non-tangential Stolz angle  at the point
$e^{i\theta}$. 
 A set $\Lambda$ is called  PT-sampling\footnote{For Pascal Thomas.} in $H^2(\bD )$ if
 $\|M_\Lambda(f)\|_{L^2} \gtrsim \|f\|_2$ for all $f\in H^2(\bD )$. 

Thomas  proves that a set $\Lambda$ is sampling for $H^2$
if and only if it norming for $H^\infty$ which was  geometrically
described by Brown, Shields and Zeller in \cite{BSZ60} 
by the property that the nontangential limit set of
$\Lambda$ must be of full measure in $\mathbb T$. This alternative
notion of sampling was inspired by a corresponding alternative
definition of interpolating sequences in the Hardy space by Bruna,
Nicolau and {\O}yma \cite{BNO96}.

 The relation between \mz\ families and
PT-sampling sets for $H^2$  is not clear. We only  mention that there
is no analog of Theorem~\ref{mzsampa} for PT-sampling:  Consider the
\mz\ family $\Lambda _n$ in the example \eqref{eq:27}. Its weak limit
in $\bD $ is just $\{0\}$, which is obviously not PT-sampling. Its
weak limit in $\bC $ is $\{0\} \cup \partial \bD $,  and this not even
covered by  the definition of 
PT-sampling.   We have not pursued this aspect further. 

\end{document}

%% file: macros.tex
\newtheorem{tm}{Theorem}[section]    
\newtheorem{lemma}[tm]{Lemma}
\newtheorem{prop}[tm]{Proposition}
\newtheorem{cor}[tm]{Corollary}

\theoremstyle{definition}
\newtheorem{df}{Definition}[section]

\newcommand{\beqa}{\begin{eqnarray*}}
\newcommand{\eeqa}{\end{eqnarray*}}

\newcommand{\field}[1]{\mathbb{#1}}
\newcommand{\bN}{\field{N}}        
\newcommand{\bC}{\field{C}}        
\newcommand{\bT}{\field{T}}        %
\newcommand{\bD}{\field{D}}

\def\cH{\mathcal{ H}}

\def\cP{\mathcal{ P}}

\def\<{\left<}
\def\>{\right>}

\def\inv{^{-1}}

\def\mv1{M_v^1}

